\documentclass[10pt,a4paper]{article}
\usepackage{xypic}
\usepackage{amssymb}
\usepackage{amsthm}
\usepackage{bm}
\usepackage{latexsym}
\usepackage{amsmath}
\usepackage{eufrak}
\usepackage{mathrsfs}
\usepackage{amscd}
\usepackage[all]{xy}
\usepackage[usenames]{color}

\usepackage{xcolor,colortbl}
\usepackage{soul}

\usepackage{graphicx}

\usepackage{color}
\usepackage{amscd}
\theoremstyle{plain}

\newtheorem{theorem}{Theorem}[section]
\newtheorem{proposition}[theorem]{Proposition}
\newtheorem{lemma}[theorem]{Lemma}

\theoremstyle{definition}

\newcommand{\appsection}[1]{\let\oldthesection\thesection
\renewcommand{\thesection}{Appendix \oldthesection}
\section{#1}\let\thesection\oldthesection}

\theoremstyle{remark}

\newtheorem{remark}[theorem]{Remark}

\def\Z{{\mathbb{Z}}}

\def\Q{{\mathbb{Q}}}
\def\C{{\mathbb{C}}}
\def\P{{\mathbb{P}}}

\def\O{{\mathcal{O}}}

\newcommand{\Proj}{\operatorname{Proj}}
\newcommand{\HHom}{\mathcal{H}om}

\begin{document}
\bibliographystyle{amsplain}

\title{On degenerations of $\Z/2$-Godeaux surfaces}
\author{\textrm{Eduardo Dias, Carlos Rito, Giancarlo Urz\'ua}}
\date{}

\maketitle

\begin{abstract} 
We compute equations for the Coughlan's family in \cite{CoughlanGodeaux} of Godeaux surfaces with torsion $\Z/2$, which we call $\Z/2$-Godeaux surfaces, and we show that it is (at most) 7 dimensional. We classify non-rational KSBA degenerations $W$ of $\Z/2$-Godeaux surfaces with one Wahl singularity, showing that $W$ is birational to particular either Enriques surfaces, or $D_{2,n}$ elliptic surfaces,
with $n=3,4$ or $6$. We present examples for all possibilities in the first case, and for $n=3,4$ in the second.
\end{abstract}

\tableofcontents

\section{Introduction} \label{s0}

Smooth minimal complex projective surfaces of general type with the lowest possible numerical invariants,
namely geometric genus $p_g=0$ and self-intersection of the canonical divisor $K^2=1$,
are known to exist since Godeaux' construction in 1931 \cite{God31}. His surface has topological fundamental group $\mathbb Z/5$. Surfaces of general type with $K^2=1, p_g=0$ are called numerical Godeaux surfaces. Miyaoka \cite{Miy76} showed that the order of their torsion group is at most 5, and Reid \cite{Rei78} excluded the case $(\mathbb Z/2)^2$, so
their possible torsion groups are $\mathbb Z/n$ with $1\leq n\leq 5$. All of them are realizable. Reid constructed the moduli space for the cases $n=5,4,3$,
and it follows from his work that the topological fundamental group coincides with the torsion group for $n=5,4$.
Urz\'ua and Coughlan \cite{CU18} showed that the same happens for $n=3$. In those three cases, the moduli space is unirational and irreducible of dimension 8. Reid conjectured that the same should happen for
numerical Godeaux surfaces with torsion $\mathbb Z/2$ and with no torsion. Both cases remain a challenge as far as we know; it is not even known if the topological fundamental groups are indeed $\Z/2$ and trivial in these two cases. Several authors have worked on these surfaces, and there are some unrelated constructions of some components of the moduli space. (See e.g. \cite{CP00}, \cite[\S 6]{CataneseY}, \cite{BCP11} for a survey on $p_g=0$ surfaces and various references, \cite{RTU17}). In the case of $\Z/2$-Godeaux surfaces, Catanese and Debarre \cite{CataneseY} show that their \'etale double covers
are surfaces with birational bicanonical map and hyperelliptic canonical curve,
and they do a general study of its canonical ring.
Coughlan \cite{CoughlanGodeaux} gives the construction of a family depending on 8 parameters.

In this paper, we implement Coughlan's construction, overcoming some computational difficulties,
and we obtain explicit equations for his family of surfaces.
We show that it depends on at most $7$ parameters, so the problem of classification of $\mathbb Z/2$-Godeaux surfaces is still wide open.
(We recall that deformation theory has been used to show the existence of 8-dimensional components of $\mathbb Z/2$-Godeaux surfaces, see e.g. \cite{Werner8dim}, \cite{KLP12}, and Remark \ref{qgor}.)
Each surface is embedded in the projective space $\mathbb P(1,2,2,2,3,3,3,3,4)$,
and we give equations for the embedding by the tricanonical map into $\mathbb P^7$,
as well as the image by the bicanonical map, an octic surface in $\mathbb P^3$.
Moreover, we show that the \'etale coverings of Coughlan's surfaces belong to the 16-dimensional component
$\mathscr M_E$ described in \cite[\S 5]{CataneseY}, thus their topological fundamental group is $\mathbb Z/2$.

We also classify deformations to non-rational surfaces $W$ with a unique Wahl singularity, and ample canonical class. Hence these surfaces $W$ belong to the Koll\'ar-Shepher-Barron--Alexeev (KSBA) compactification of the moduli space of Godeaux surfaces \cite{KSB88,A94} (see \cite{Hack11}). The relevance of stable surfaces with one Wahl singularity is that, under no obstructions in deformation, they represent boundary divisors in the KSBA compactification (see \cite[\S 9]{Hack11}), and they are abundant (see e.g. \cite{SU16}). What allows us to classify is the recent work \cite{RU19} which optimally bounds Wahl singularities in stable surfaces with one singularity. Using that work and the particular situation of degenerations of $\Z/2$-Godeaux surfaces, in this paper we show that the smooth minimal model of $W$ is a particular either Enriques surface or $D_{2,n}$ elliptic surface, with $n\in {3,4,6}$. We give a complete list of the geometric possibilities and the singularities that may occur.

The description of an Enriques surface as a double plane makes the construction of examples for this case simpler, and in fact we give constructions for all cases in the list. Although, we also use the explicit MMP in \cite{HTU17,U16} to show existence for some cases.

The case of $D_{2,n}$ elliptic surfaces is much harder, explicit constructions are difficult to obtain, so we take a different approach. We search for such degenerations using our equations of Coughlan's family of $\mathbb Z/2$-Godeaux surfaces.
The method is to study random surfaces, working over finite fields, in order to get ideas of where the interesting cases may be, then try
to construct it over the complex numbers. We explicitly obtain codimension 1 families of $D_{2,4}$ and $D_{2,3}$ elliptic surfaces (containing a $(-4)$-curve).
The existence of the first case can also be proved via MMP \cite{HTU17,U16} and we indicate how, but of course this is not explicit. That case appears in several constructions by deformations, suggesting that the irreducibility of the moduli space may hold.
The second case is more interesting because one can show
the existence of $D_{2,3}$ surfaces with a $(-4)$-curve inside, whose contraction can be $\Q$-Gorenstein smoothed to simply connected Godeaux surfaces (see e.g. \cite[\S 5]{U16}).

The paper is organized as follows. In Section \ref{ListPossibilities} we give a complete list of possibilities for the deformations to non-rational surfaces with one Wahl singularity that may occur. In Sections \ref{Enriques}, \ref{Realizations} we construct several examples of such degenerations:
all possible cases with $W$ an Enriques surface, with explicit constructions; and two different cases with $W$ a $D_{2,4}$ surface, using deformation theory and MMP.
In Section \ref{CoughlanFamily} we describe how to find explicit equations for Coughlan's family of $\mathbb Z/2$-Godeaux surfaces, using computer algebra,
and we show that it is (at most) 7-dimensional. Finally in Section \ref{CoughlanDnm} we explain how to find, in this family, equations for surfaces in the cases where $W$ is a $D_{2,4}$ or $D_{2,3}$ surface.

\subsubsection*{Notation}

\noindent  

\begin{itemize}


\item
A $(-m)$-curve is a curve isomorphic to $\P^1$ with self-intersection $-m$.

\item
A $D_{n,m}$ surface is a smooth projective surface with an elliptic fibration over $\mathbb P^1$ which has two fibres of multiplicities $n,m$, and $p_g=0$. The fundamental group of $D_{n,m}$ is $\Z/\text{gcd}(n,m)$ (see e.g. \cite[\S 3]{Dolg88}). 

\item
A $\Z/2$-Godeaux surface is a smooth minimal projective surface with $p_g=0$, $K^2=1$, and $\pi_1^{\text{\'et}}=\Z/2$ (which is equivalent to have torsion group $\Z/2$).

\item
If $\phi \colon X \to W$ is a birational morphism, then exc$(\phi)$ is the exceptional divisor.
The strict transform of an irreducible curve $\Gamma$ in $W$ will be denoted by $\Gamma$ again.

\item
A cyclic quotient singularity $Y$, denoted by $\frac{1}{m}(1,q)$, is a germ at the origin of the quotient of $\C^2$ by the action of
$\mu_m$ given by $(x,y)\mapsto (\mu_m x, \mu_m^q y)$, where $\mu_m$ is a primitive $m$-th root of $1$, and $q$ is an integer with $0<q<m$ and gcd$(q,m)=1$.
If $\sigma \colon \widetilde{Y} \rightarrow Y$ is the minimal resolution of $Y$, then the exceptional curves $E_i=\P^1$ of $\sigma$, with $1 \leq i \leq s$,
form a chain such that $E_i^2=-b_i$ where $ \frac{m}{q} = [b_1, \ldots ,b_s]$ is the Hirzebruch-Jung continued fraction.
Commonly we will refer to exc$(\sigma)$ as $[b_1, \ldots ,b_s]$.

\item
The Kodaira dimension of $X$ is denoted by $\kappa(X)$.

\item
A KSBA surface in this paper is a normal projective surface with log-canonical singularities and ample canonical class \cite{KSB88}.
\end{itemize}

\subsubsection*{Acknowledgments}

We thank Stephen Coughlan for useful conversations related to this paper. The first and second authors were supported by FCT (Portugal) under the project PTDC/MAT-GEO/2823/2014
and by CMUP (UID/ MAT/00144/2019), which is funded by FCT with national (MCTES) and European structural funds through the programs FEDER,
under the partnership agreement PT2020. The second author was supported by the fellowship SFRH/BPD/111131/2015. The third author thanks FONDECYT for support from the regular grant 1190066. The second and third authors thank Pontificia Universidad Cat´\'olica de Chile and Universidade do Porto for the hospitality during their visits in December 2018 and February 2019.

\section{Non-rational degenerations with one Wahl singularity}

A Wahl singularity is a cyclic quotient singularity of the type $1/n^2(1,na-1)$ with $0<a<n$ coprime. Equivalently, they are precisely the cyclic quotient singularities which admit a smoothing with Milnor number equal to zero. KSBA surfaces with one Wahl singularity turn out to be abundant in the closure of the moduli space of surfaces of general type. When in addition there are no local-to-global obstructions, these surfaces represent divisors in the KSBA compactification (see \cite[\S 4]{Hack11}). In this section we classify all possible degenerations of $\Z/2$-Godeaux surfaces into non-rational KSBA surfaces with one Wahl singularity. The main tool is \cite{RU19}, where we can find explicit optimal bounds for Wahl singularities and some useful features for particularly ``small" cases.

\subsection{List of possibilities}\label{ListPossibilities}

\begin{figure}[htbp]
\centering
\includegraphics[width=13cm]{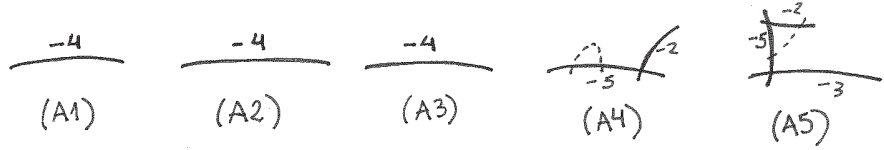}
\caption{Options for $\kappa=1$} \label{f1}
\end{figure}

\begin{figure}[htbp]
\centering
\includegraphics[width=13cm]{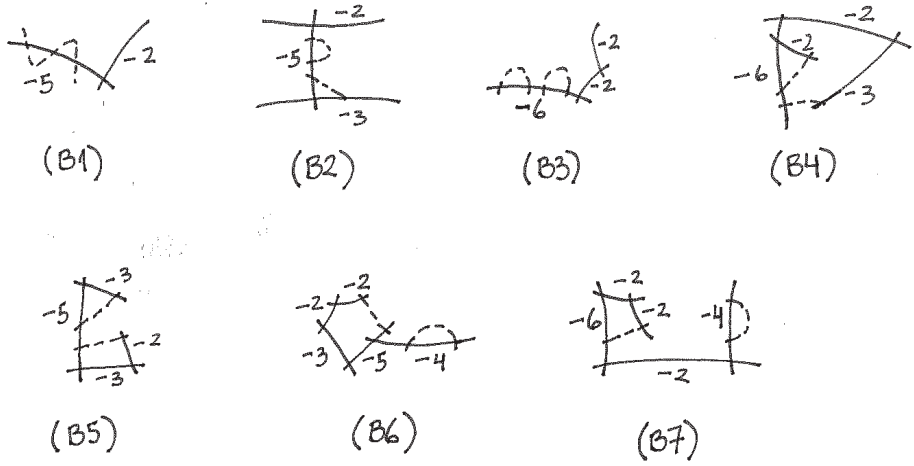}
\caption{Options for $\kappa=0$} \label{f2}
\end{figure}

\begin{theorem}
Let $W$ be a $\Q$-Gorenstein degeneration of a $\Z/2$-Godeaux surface which has one Wahl singularity and $K_W$ ample. If $\phi \colon X \to W$ is the minimal resolution and $X$ is not rational, then $X$ belongs to the following list:

\begin{itemize}
\item[A.] $\kappa(X)=1$

\begin{enumerate}
\item[$(1)$] The surface $X$ is a $D_{2,3}$, and exc$(\phi)=[4]$. 

\item[$(2)$] The surface $X$ is a $D_{2,6}$, and exc$(\phi)=[4]$.

\item[$(3)$] The surface $X$ is a $D_{2,4}$, and exc$(\phi)=[4]$.

\item[$(4)$] The surface $X$ is the blow-up at one point of a $D_{2,4}$, and exc$(\phi)=[5,2]$. The $(-1)$-curve intersects the $(-5)$-curve with multiplicity $2$.

\item[$(5)$] The surface $X$ is the blow-up of a $D_{2,4}$ twice at the node of the multiplicity four $I_1$ fiber, and exc$(\phi)=[3,5,2]$. The surface $D_{2,4}$ contains a $(-3)$-curve which is a $4$-section.

\end{enumerate}

\item[B.] $\kappa(X)=0$, and $X$ is an Enriques surface blown-up 

\begin{enumerate}
\item[$(1)$] once, and exc$(\phi)=[5,2]$. The $(-1)$-curve intersects the $(-5)$-curve with multiplicity $3$.

\item[$(2)$] twice, and exc$(\phi)=[2,5,3]$. There is one $(-1)$-curve touching the $(-5)$-curve with multiplicity $2$, and there is another $(-1)$-curve intersecting the $(-5)$-curve and the $(-3)$-curve at one point. 

\item[$(3)$] twice, and exc$(\phi)=[6,2,2]$. There are two disjoint $(-1)$-curves intersecting the $(-6)$-curve with multiplicity $2$ each.

\item[$(4)$] three times, and exc$(\phi)=[2,6,2,3]$. There is a $(-1)$-curve intersecting the first $(-2)$-curve and the $(-6)$-curve at one point each, and there is a $(-1)$-curve intersecting the $(-6)$-curve and the $(-3)$-curve at one point each.

\item[$(5)$] three times, and exc$(\phi)=[3,5,3,2]$. There is a $(-1)$-curve intersecting the first $(-3)$-curve and the $(-5)$-curve at one point each, and there is a $(-1)$-curve intersecting the $(-5)$-curve and the $(-2)$-curve at one point each.

\item[$(6)$] four times, and exc$(\phi)=[2,2,3,5,4]$. There is a $(-1)$-curve intersecting the first $(-2)$-curve and the $(-5)$-curve at one point each, and there is a $(-1)$-curve intersecting the $(-4)$-curve with multiplicity $2$.

\item[$(7)$] four times, and exc$(\phi)=[2,2,6,2,4]$. There is a $(-1)$-curve intersecting the first $(-2)$-curve and the $(-6)$-curve at one point each, and there is a $(-1)$-curve intersecting the $(-4)$-curve with multiplicity $2$.

\end{enumerate}

\end{itemize}
Moreover, all cases do exist, except possibly A$2$ and A$5$.
\label{possible}
\end{theorem}

\begin{proof}

First, by \cite[Proposition 2.2]{RU19} and our hypothesis ($K_W^2=1$ and $W$ non-rational), we have that $X$ is the blow-up of either an elliptic surface with $q=0$ or an Enriques surface. Note that $p_g(W)=0$ as well, because $W$ is a $\Q$-Gorenstein degeneration of a $\Z/2$-Godeaux surface $Z$.

\vspace{0.2cm}
{\bf Say that $\kappa(X)=1$.} Let $\pi \colon X \to S$ be the blow-down to a minimal surface $S$. Hence, in our situation, $S$ has an elliptic fibration $S \to \P^1$.

As in the proof of \cite[Proposition 6.1]{RTU17}, we have that $$\pi_1^{\text{\'et}}(Z) \to \pi_1^{\text{\'et}}(W)$$ is surjective, $\pi_1^{\text{\'et}}(W) \simeq \pi_1^{\text{\'et}}(X)$, and $\pi_1(X)$ is residually finite, and so $\pi_1(X)$ could be trivial or $\Z/2$. As the fundamental group is finite and $\kappa(S)=1$, by \cite[Corollary p.146]{Dolg88} we have that the elliptic surface $S$ must have two multiple fibres and so it is a $D_{n,m}$. Since $\pi_1(X)$ could be trivial or $\Z/2$, we have that gcd$(n,m)=1$ or $2$ respectively. The canonical class formula gives $K_S \sim -F + (n-1)F_n + (m-1)F_m$ where $F$ is a general fibre, and the divisors $F_n$, $F_m$ are reduced fibres so that $F \sim nF_n \sim mF_m$. 

On the other hand, by \cite[Theorem 2.15]{RU19} we have that the exceptional divisors in $X$ could be $[4]$, $[5,2]$, $[6,2,2]$, $[2,5,3]$. We now check case by case:

\vspace{0.2cm}
$[4]$: Then $X=S$. Let $C$ be the $(-4)$-curve. Then $K_S \cdot C=2$ gives restrictions on gcd$(n,m)$. The only possible pairs $(n,m)$ are $(2,3)$, $(2,4)$ and $(2,6)$.

\vspace{0.2cm}
$[5,2]$: Then $X \to S$ is the blow-up at one point. The $(-1)$-curve cannot touch the $(-2)$-curve, and so it must touch the $(-5)$-curve with multiplicity $2$. (It cannot just intersect it at one point since $K_W$ is ample, and for multiplicity $>2$ it would be trivial or negative for $K_S$.) So in $S$ the $(-5)$-curve becomes a curve $\Gamma$ such that $\Gamma \cdot K_S=1$. Then we use the canonical class formula and gcd$(n,m)=1$ or $2$ to get that $(n,m)=(2,4)$ or $(2,3)$ only. In the case of $(2,3)$ we have a simply connected surface. Then by \cite[Corollary 1.2.4]{H16} and since the index of $[5,2]$ is $3$, we obtain an exact sequence $\Z/3 \to \Z/2 \to 0$, which is a contradiction. Therefore the only possible case is $(2,4)$.

\vspace{0.2cm}
$[6,2,2]$: Then $X \to S$ is the composition of two blow-ups. According to \cite[Corollaries 2.12, 2.13 and Theorem 2.15]{RU19} this case can only happen with a $(-1)$-curve which forms a long diagram of type I or II (see \cite{RU19} for the definition). But then there is a $(-1)$-curve intersecting one of the $(-2)$-curves transversally at only one point, and this is a contradiction with the number of blow-ups from $S$.

\vspace{0.2cm}
$[2,5,3]$: Similarly, according to \cite[Corollaries 2.12, 2.13 and Theorem 2.15]{RU19} this can only happen as $X \to S$ blow-up twice, where there is a $(-1)$-curve in $X$ intersecting once the $(-2)$-curve and once the $(-5)$-curve, disjoint from the $(-3)$-curve. But then the $(-3)$-curve in $S=D_{n,m}$ intersects a nodal rational curve at one point. Then, by using adjunction, we obtain that the only possible pair is $(n,m)=(2,4)$ where the multiplicity $4$ fiber is the $I_1$ image under $X \to S$ of the $(-5)$-curve.\footnote{This was not considered in \cite[Theorem 3.2]{RU19}, but it is in the arXiv corrected version.}


\vspace{0.2cm}
{\bf Say now that $\kappa(X)=0$.} Let $\pi \colon X \to S$ be the blow-down to an Enriques surface $S$. By \cite[Corollaries 2.12, 2.13 and Theorem 2.15]{RU19} we have that the exceptional divisor in $X$ could have at most $5$ $\P^1$'s. The case of $5$ $\P^1$'s was classified in \cite[Theorem 3.1]{RU19}, and it gives precisely the cases (6) and (7) in the list above. Thus we now check case by case when we have at most $4$ $\P^1$'s:

\vspace{0.2cm}
$[4]$: This case is impossible since $X=S$ and $K_S \equiv 0$.  

\vspace{0.2cm}
$[5,2]$: We have that $X \to S$ is the blow-up at one point. The $(-1)$-curve cannot touch the $(-2)$-curve. The only possibility then is that it touches the $(-5)$-curve with multiplicity $3$.

\vspace{0.2cm}
$[6,2,2]$: Here $X \to S$ contracts two $(-1)$-curves. One checks that a $(-1)$-curve must be disjoint from the $(-2)$-curves. Since $K_S \equiv 0$, the only possible situation is to have two disjoint $(-1)$-curves intersecting the $(-6)$-curve at two points each.

\vspace{0.2cm}
$[2,5,3]$: In this case $X \to S$ is blow-up twice. The $(-1)$-curve cannot touch the $(-2)$-curve. Since we have a $(-3)$-curve in $X$, we need a $(-1)$-curve touching it once. Since $K_W$ is ample, it must intersect the $(-5)$-curve. It can only be at one point, and there must exist another $(-1)$-curve intersecting the $(-5)$-curve with multiplicity $2$. 

\vspace{0.2cm}
For the next cases, it can only be the situation of a long diagram of type I or II. The map $X \to S$ is a blow-up three times.

\vspace{0.2cm}
$[7,2,2,2]$: It is not possible, since we would have $4$ blow-downs.

\vspace{0.2cm}
$[2,6,2,3]$: There is a $(-1)$-curve intersecting the first $(-2)$-curve and the $(-6)$-curve at one point each. After contracting it and the new $(-1)$-curve from the $(-2)$-curve, we obtain a nodal rational curve with self-intersection $-1$. We still have a $(-3)$-curve, and so a new $(-1)$-curve is needed intersecting it at one point, and also the nodal $(-1)$-curve.

\vspace{0.2cm}
$[2,2,5,4]$: Long diagrams of type I or II here are not possible, just using that $K_S \equiv 0$. 

\vspace{0.2cm}
$[3,5,3,2]$: Here the long diagram gives a $(-1)$-curve intersecting the $(-2)$-curve and the $(-5)$-curve at one point each. After that, one can check that there must be a $(-1)$-curve intersecting the first $(-3)$-curve with the $(-5)$-curve.

The existence of such surfaces will be proved in the next sub-section.
\end{proof}

\begin{remark}
For simply-connected Godeaux surfaces the analogue non-rational list contains only two possible surfaces: either a $D_{2,3}$ with Exc$(\phi)=[4]$, or the blow-up at one point of a $D_{2,3}$ with Exc$(\phi)=[5,2]$ and a $(-1)$-curve intersecting the $(-5)$-curve with multiplicity $2$. Both are realizable (see e.g. \cite[Table in p.666]{SU16}), and give divisors in the KSBA compactification of the moduli space.
\end{remark}

\begin{remark}
It is not clear how to optimally bound Wahl singularities in rational surfaces.
As far as we know, there is no written example of a rational degeneration $W$ of $\Z/2$-Godeaux surfaces in the literature.
(We believe they exist in Coughlan's family of $\Z/2$-Godeaux surfaces, but the computations involved in order to describe them are terribly slow.)
However for simply-connected Godeaux surfaces there are many examples (see e.g. \cite[Table in p.666]{SU16}, where there are $30$ examples). We note that in this rational case the index of the Wahl singularity for a $\Z/2$-Godeaux degeneration must be even because of \cite[Corollary 1.2.4]{H16}. 
\end{remark}

\subsection{Enriques double planes}\label{Enriques}

The following construction of an Enriques surface as the smooth minimal model of a double plane is well-known, see e.g. \cite[\S IV.9]{CoDo89}.

Consider lines $L_1,L_2\subset\mathbb P^2$ meeting at a point $p_0,$ and take points $p_1\in L_1,$ $p_2\in L_2,$ $p_1,p_2\not= p_0.$
Let $B$ be a sextic plane curve with a node at $p_0,$ a tacnode at $p_i$ with branches tangent to the line $L_i,$ for $i=1,2,$
and at most other negligible singularities.
Let $\pi:X\rightarrow\mathbb P^2$ be the blow-up that resolves the singularities of the curve
$B+L_1+L_2$. For $i=1,2,$ let $\widehat L_i$ be the strict transform of $L_i$ and $E_0,E_i,E_i'$ be the exceptional curves such that
the total transform of $L_i$ is $E_0+\widehat L_i+2E_i'+E_i$ (we have $E_i^2=-2,$ $E_i'^2=-1$).
Let $S'\rightarrow X$ be the double cover with branch curve
$$\overline B:=\widehat B +\widehat L_1+\widehat L_2+E_1+E_2,$$ where $\widehat B$ is the strict
transform of $B.$ Let $S$ be the minimal model of $S',$ which is obtained by contracting the four $(-1)$-curves that are the preimage of
$\widehat L_1+\widehat L_2+E_1+E_2.$ The surface $S$ is an Enriques surface.

\begin{figure}[h]
  \includegraphics[width=\linewidth]{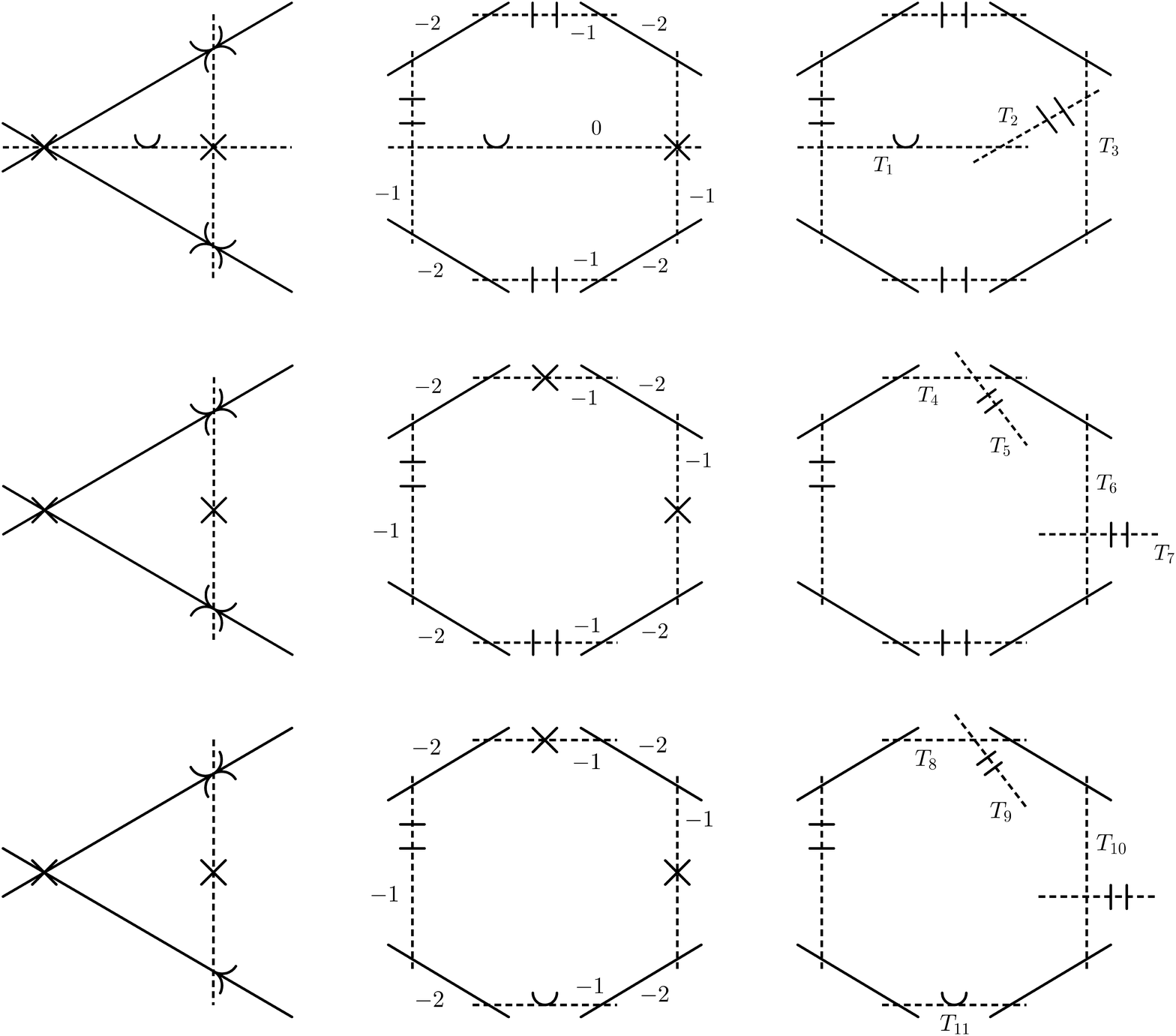}
  \caption{Sextic curves and their resolution.} \label{sextics}
\end{figure}

Consider the three Enriques surfaces corresponding to branch curves $\overline B\subset X$ as in Figure \ref{sextics},
where from left to right we blow-up $\mathbb P^2$ until we resolve the singularities of the curve (dotted curves are not in the branch curve).
Note that the existence of such curves is not surprising, because we are imposing at most 20 conditions to a linear system of dimension 27. 
We give explicit equations in an arXiv ancillary file. 

These smooth Enriques surfaces have a configuration of rational curves as in Figure \ref{f3},
with the correspondences
\begin{itemize}
\item[(i)] $T_1\longleftrightarrow C+D,\ \ T_2\longleftrightarrow B,\ \ T_3\longleftrightarrow A,$
\item[(ii)] $T_4\longleftrightarrow B,\ \ T_5\longleftrightarrow A,\ \ T_6\longleftrightarrow D,\ \ T_7\longleftrightarrow C,$
\item[(iii)] $T_8\longleftrightarrow B,\ \ T_9\longleftrightarrow A,\ \ T_{10}\longleftrightarrow C,\ \ T_{11}\longleftrightarrow D.$
\end{itemize}

\subsection{Realizations of degenerations}\label{Realizations}

In this section we discuss the realization of the possibilities in Theorem \ref{possible}. Below we follow the numeration in that theorem. We do not know about existence for the possibilities \textbf{(A2)} and \textbf{(A5)}. The case \textbf{(A2)} is in the classification of degenerations of Godeaux surfaces with one $\frac{1}{4}(1,1)$ singularity in \cite{K14}, but there was no construction (see \cite[Remark 2.11]{K14}). 

\vspace{0.2cm}
\textbf{(A1)} 

It can be realized using our equations of Coughlan's family of surfaces. For the details see Section \ref{D23}.

These degenerations are particularly interesting since we have the same $\Q$-Gorenstein degenerations via simply connected Godeaux surfaces (see e.g. \cite[\S 5]{U16}). 

\vspace{0.2cm}
\textbf{(A3)} 

This possibility can be realized using \cite[Example 1]{KLP12}. The singular surface $W$ constructed for that example has $4$ singularities: two $[2,3,2,4]$ and two $[4]$. It also has no local-to-global obstructions to deform, and it is proved that a $\Q$-Gorenstein smoothing is a Godeaux surface with $\pi_1 \simeq \Z/2$. We realize a surface in (A3) as the minimal resolution of a $\Q$-Gorenstein smoothing of all singularities in $W$ except one $[4]$. To show that it is indeed a $D_{2,4}$ with a $(-4)$-curve inside, we use the explicit MMP in \cite{HTU17} (see also \cite{U16}). We note that it is not a trivial computation since we have $3$ possibilities for $D_{n,m}$ here. At the end, it is a $D_{2,4}$ because it comes from a $\Q$-Gorenstein smoothing ``over" a multiplicity $2$ fiber for the singularity $[4]$. This example gives a divisor in the KSBA moduli space, whose general member is a $D_{2,4}$ with the $(-4)$-curve contracted.

Also, it can be realized explicitly using our equations of Coughlan's family of surfaces. For the details see Section \ref{D24}.

\vspace{0.2cm}
\textbf{(A4)} 

Take again the singular surface $W$ in \cite[Example 1]{KLP12} but now we $\Q$-Gorenstein smooth all singularities in $W$ except one $[2,3,2,4]$. By the explicit MMP in \cite{HTU17} we obtain that the minimal resolution of $[2,3,2,4]$ is the blow-up of a $D_{2,4}$ at one point, where the $(-1)$-curve connects the $(-3)$-curve with the $(-4)$-curve. We recall that the $M$-resolution of $[2,3,2,4]$ is the partial resolution $[2,5]-1-[2,5]-1-[2,5]$ which also has no-local-to-global obstructions. Then, we just keep one $[2,5]$ in a $\Q$-Gorenstein, smoothing all the rest to obtain the surface we are looking for. Its minimal resolution corresponds to (A4). As in (A3), this example gives a divisor in the KSBA moduli space.

\begin{figure}[htbp]
\centering
\includegraphics[width=11cm]{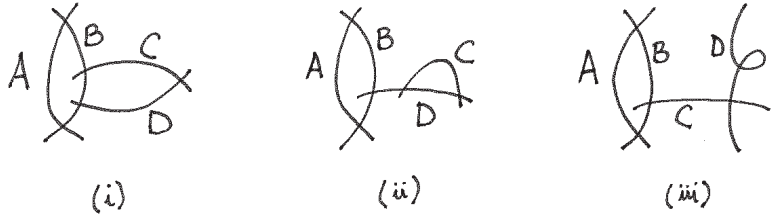}
\caption{Key configurations of Enrique type} 
\label{f3}
\end{figure}

\vspace{0.2cm}
\textbf{(B1)} and \textbf{(B4)} 

From Section \ref{Enriques}, there exists an Enriques surface $S$ which has the configuration of smooth rational curves $A,B,C,D$ shown in Figure \ref{f3} part (i).
Let $\pi \colon X \to S$ be the blow-up of $S$ five times, so that the configuration $A,B,C,D$ is transformed into the configuration in Figure \ref{f4}, where the $E_i$ are the ordered exceptional curves. Hence $E_1^2=E_3^2=-2$, and $E_2^2=E_4^2=E_5^2=-1$.

\begin{figure}[htbp]
\centering
\includegraphics[width=5cm]{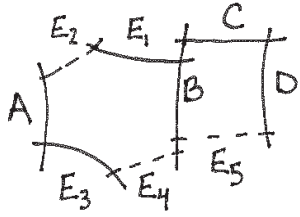}
\caption{The surface $X$ for cases (B1) and (B4)} 
\label{f4}
\end{figure}

We get Wahl chains $[E_1,B,C,D]=[2,6,2,3]$ and $[A,E_3]=[5,2]$. Let $\phi \colon X \to W$ be the contraction of both of them. The normal projective surface $W$ has two Wahl singularities $\frac{1}{7^2}(1,20)$ and $\frac{1}{3^2}(1,2)$. The canonical class $K_W$ is ample since $\phi^*(K_W)$ can be written $\Q$-effectively using only curves in Figure \ref{f4}, and so we check ampleness through the intersections $\phi^*(K_W).E_i >0$ for $i=2,4,5$.
\footnote{We may find ADE configurations disjoint from $A,B,C,D$ in $S$ which would intersect $K_W$ trivially. If that happens, one can always smooth them up so that $K$ for the resulting surface is ample.}

We now show that there are no local-to-global obstructions to deform $W$. For that it is enough to show that $H^2(S, T_S(-\log(A+B+C+D)))=0$, following the well-known strategy from \cite{LP07}. We will use the following lemma (see \cite[Section 3.1]{RU19}).

\begin{lemma}
Let $f \colon S' \to S$ be the \'etale double cover induced by the relation $2K_S \sim 0$. Let $\Gamma_1,\ldots, \Gamma_r$ be a simple normal crossings divisor in $S$. Then $$f_*\big(T_{\bar S}\big(-\log\big(\sum_i f^* \Gamma_i \big)\big)\big)=T_S\big(-\log\big(\sum_i \Gamma_i \big)\big)\oplus T_S\big(-\log\big(\sum_i \Gamma_i \big)\big)\big(-K_S \big),$$ and so $$H^0(S',\Omega_{S'}^1 \big(\log\big(\sum_i f^* \Gamma_i \big)\big)=H^2(S,T_S\big(-\log\big(\sum_i \Gamma_i \big)\big) \oplus H^0(S,\Omega_S^1\big(\log \sum_i \Gamma_i \big) ).$$ In particular, if the curves $\{f^*(\Gamma_i)\}_{i=1}^r$ are numerically independent, then $$H^0(S',\Omega_{S'}^1 \big(\log\big(\sum_i f^* \Gamma_i \big)\big)=0,$$ and so $H^2(S,T_S\big(-\log\big(\sum_i \Gamma_i \big)\big)=0$.
\label{trick}
\end{lemma}

By Lemma \ref{trick}, we only need to check that $f^*(A+B+C+D)$ is a divisor supported in numerically independent curves. For that we compute the corresponding intersection matrix and check that the determinant is not zero. Then we have no local-to-global obstructions, and we consider a $\Q$-Gorenstein smoothing of $W$. Since we have $E_2 \simeq \P^1$ connecting the ends of the Wahl chains and the indexes of the singularities are coprime, we obtain that the general fiber has fundamental group isomorphic to $\Z/2$. One also has $K_W^2=1$, and $p_g=q=0$, and so we have $\Z/2$-Godeaux surfaces as general fibers.

To obtain examples of types (B1) and (B4), we consider the minimal resolution of the partial $\Q$-Gorenstein smoothing of $[2,6,2,3]$ or $[5,2]$, respectively. To check that they are indeed blow-ups of Enriques surfaces, we run the explicit MMP in \cite{HTU17}. For each of the singularities we obtain a divisor in the KSBA compactification of the moduli space of $\Z/2$-Godeaux surfaces. Both of these examples are new in the literature.

\vspace{0.2cm}
\textbf{(B2) and (B5)} 

From Section \ref{Enriques}, there exists an Enriques surface $S$ which has the configuration of smooth rational curves $A,B,C,D$ shown in Figure \ref{f3} part (ii).

\begin{figure}[htbp]
\centering
\includegraphics[width=9cm]{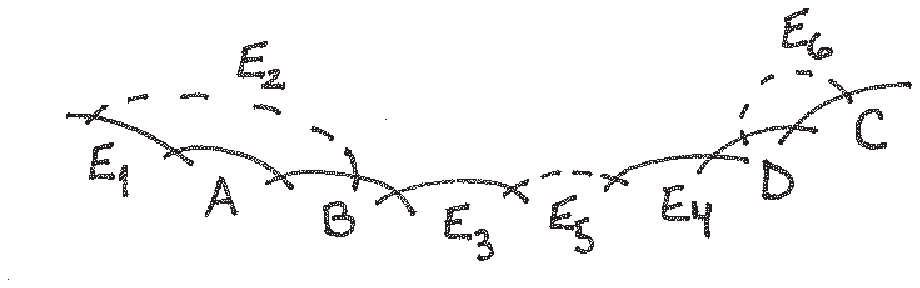}
\caption{The surface $X$ for cases (B2) and (B5)} 
\label{f5}
\end{figure}

Let $\pi \colon X \to S$ be the blow-up of $S$ six times, so that the configuration $A,B,C,D$ is transformed into the configuration in Figure \ref{f5}, where the $E_i$ are the ordered exceptional curves. Hence $E_3^2=-3$, $E_1^2=E_4^2=-2$, and $E_2^2=E_5^2=E_6^2=-1$. We get Wahl chains $[E_1,A,B,E_3]=[2,3,5,3]$ and $[E_4,D,C]=[2,5,3]$. Let $\phi \colon X \to W$ be the contraction of both of them. The normal projective surface $W$ has two Wahl singularities $\frac{1}{8^2}(1,23)$ and $\frac{1}{5^2}(1,9)$. The canonical class $K_W$ is ample since $\phi^*(K_W)$ can be written $\Q$-effectively using only curves in Figure \ref{f5}, and so we check ampleness through the intersections $\phi^*(K_W).E_i >0$ for $i=2,5,6$. \footnote{Just as in the previous example, zero curves for $K_W$ do not matter.} 

We now show that there are no local-to-global obstructions to deform $W$. As done above, for that it is enough to show that $H^2(S, T_S(-\log(A+B+C+D)))=0$. We use again Lemma \ref{trick}, and so we only need to check that $f^*(A+B+C+D)$ is a divisor supported in numerically independent curves. For that we compute the corresponding intersection matrix and check that the determinant is not zero. Then we have no local-to-global obstructions, and we consider a $\Q$-Gorenstein smoothing of $W$. Since we have $E_5 \simeq \P^1$ connecting the ends of the Wahl chains and the indices of the singularities are coprime, we obtain that the general fiber has fundamental group isomorphic to $\Z/2$. One also has $K_W^2=1$, and $p_g=q=0$, and so we have $\Z/2$-Godeaux surfaces as general fibers.

To obtain examples of types (B2) and (B5), we consider the minimal resolution of the partial $\Q$-Gorenstein smoothing of $[2,3,5,3]$ or $[2,5,3]$, respectively. To check that they are indeed blow-ups of Enriques surfaces we run the explicit MMP in \cite{HTU17}. For each of the singularities we obtain a divisor in the KSBA compactification of the moduli space of $\Z/2$-Godeaux surfaces. Both of these examples are new in the literature.

\vspace{0.2cm}
\textbf{(B3) and (B7)} 

As in the previous examples, we first construct an Enriques surface $S$ which has the configuration of smooth rational curves $A,B,C,D$ shown in Figure \ref{f3} part (iii), see Section \ref{Enriques}.

\begin{figure}[htbp]
\centering
\includegraphics[width=7cm]{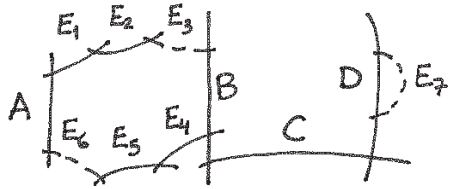}
\caption{The surface $X$ for cases (B3) and (B7)} 
\label{f6}
\end{figure}

Let $\pi \colon X \to S$ be the blow-up of $S$ seven times, so that the configuration $A,B,C,D$ is transformed into the configuration in Figure \ref{f6}, where the $E_i$ are the ordered exceptional curves. Hence $E_1^2=E_2^2=E_4^2=E_5^2=-2$, and $E_3^2=E_6^2=E_7^2=-1$. We get Wahl chains $[E_2,E_1,A]=[2,2,6]$ and $[E_5,E_4,B,C,D]=[2,2,6,2,4]$. Let $\phi \colon X \to W$ be the contraction of both of them. The normal projective surface $W$ has two Wahl singularities $\frac{1}{4^2}(1,3)$ and $\frac{1}{10^2}(1,29)$. The canonical class $K_W$ is ample since $\phi^*(K_W)$ can be written $\Q$-effectively using only curves in Figure \ref{f5}, and so we check ampleness through the intersections $\phi^*(K_W).E_i >0$ for $i=3,6,7$. All the rest of the arguments are analogues to the ones given in the last two examples, except for the computation of obstruction. Lemma \ref{trick} is used in a different way. One can prove that on the K3 surface $S'$ we have $H^0(S',\Omega_{S'}^1 \big(\log\big(\sum_i f^* \Gamma_i \big)\big)=1$, but at the same time 
$H^0(S,\Omega_{S}^1 \big(\log\big(\sum_i \Gamma_i \big)\big)=1$, and so $H^2(S,T_S\big(-\log\big(\sum_i \Gamma_i \big)\big)=0$. (It is the same argument as in case (B) of \cite[Section 3.1]{RU19}. 

For each of the singularities we obtain a divisor in the KSBA compactification of the moduli space of $\Z/2$-Godeaux surfaces. The one for $[6,2,2]$ is new in the literature, the other one is \cite[Section 3.1 (B)]{RU19}.

\vspace{0.2cm}
\textbf{(B6)} 

There is an example of this case in \cite[Section 3.1 (C)]{RU19}. It gives also a boundary divisor in the KSBA compactification of the moduli space of $\Z/2$-Godeaux surfaces. This case together with (B7) achieve the optimal upper bound for lengths of Wahl singularities in stable surfaces (see \cite[Theorem 3.1]{RU19}).

\section{Coughlan's family}\label{CoughlanFamily}

Stephen Coughlan \cite{CoughlanGodeaux} has given the construction of an irreducible family of simply connected surfaces $Y$ with invariants $p_g=1$, $q=0$, $K^2=2$ having a free action of $\Z/2$, thus producing a family of $\Z/2$-Godeaux surfaces $Y/(\Z/2)$. Here we go over his construction and implement it in order to get explicit equations for the family. This task is computationally demanding and some workarounds are needed in order to succeed.

In Section \ref{exthypk3} we give an overall resume of the method used in \cite{CoughlanGodeaux}.
In Section \ref{unprojIV} we follow the method described in \cite{ReidUnprojIV} to obtain explicit equations for the surfaces $Y$. The corresponding computations are implemented with Magma \cite{BCP}, version V2.25-2, and are available as arXiv ancillary files.

As a conclusion of this construction, we find out that one of the 8 parameters of Coughlan's family is redundant,
so his model depends on $7$ parameters, see Section \ref{count}.

We show in Section \ref{16dim} that the \'etale coverings of Coughlan's surfaces belong to the 16-dimensional component
$\mathscr M_E$ described in \cite[\S 5]{CataneseY}, thus their topological fundamental group is $\mathbb Z/2.$

\subsection{Extending hyperelliptic $K3$ surfaces}\label{exthypk3}

The description of a canonical ring for $Y$ is based on a diagram\\\\

$\begin{CD}
W'_{6,6} \subset \P(1,2^3,3^2)  @<{\rm proj}<< W \subset \P(1,2^4,3^4,4) @<<< Y\subset \P(1,2^3,3^4,4)\\
@AAA     @AAA                   @AAA\\
T'_{6,6} \subset \P(2^3,3^2)  @<{\rm proj}<< T \subset \P(2^4,3^4,4)@<<< D \subset \P(2^3,3^4,4)
\end{CD}$\\\\


\noindent where the 'proj' lines represent projections and the others are inclusions as hyperplane sections (of the correct degree),
and the varieties involved are defined below.

In \cite[pag. $72$]{extending} Reid has attempted the description of the canonical ring of $Y$ by extending the ring of its canonical curve $D$. Due to computational limitations this attempt was unsuccessful. Instead of trying to compute the extension of such a high-codimensional ideal by a variable of degree $1$, Coughlan uses simpler extensions followed by projections. This makes the varieties manageable as we will describe here.

The curve $D\subset Y$ is a hyperelliptic canonical curve section in $|K_Y|$ and its projective model is explicitly given in \cite[$\S 2$]{CoughlanGodeaux}.
It has a simple description as given by the $2\times 2$ minors of a $4\times 4$ matrix.

In \cite[$\S 3$]{CoughlanGodeaux} it is described a hyperelliptic $K3$ surface $T$ containing $D$, i.e. a $K3$ surface polarised by an ample line bundle $L$ such that the complete linear system $|L|$ contains the hyperelliptic curve $D$. In such case, $L$ determines a double cover $\pi\colon T\rightarrow Q\subset \P^3$, where $Q$ is a quadric surface, branched on a curve $C\in |-2K_Q|$. Identifying $Q$ with $\P^1\times\P^1$, the branch locus $C$ is of bidegree $(4,4)$. Assuming that it splits as $C_1+C_2$, of bidegree $(1,3)$ and $(3,1)$ respectively, the surface $T$ has $10$ nodes.

Blowing up one of the nodes in $C_1+C_2$, $P\in Q$, we get a double cover $\widetilde T\rightarrow \operatorname{Bl}_PQ$, with an exceptional divisor $E\cong \P^1$. Contracting the two $(-1)$--curves on $\operatorname{Bl}_PQ$ arising from the rulings of $Q$ we get a double cover $T'\rightarrow \P^2$ branched over two nodal cubics.

The procedure above can be seen as a projection from the point $P$.  

\begin{proposition}\cite[Prop. $3.3$]{CoughlanGodeaux}
The projection from the node $P\in T\subset \P(2^4,3^4,4)$ gives a complete intersection
$$T'_{6,6}\subset \P(2^3,3^2)=\Proj(\C[y_1,y_2,y_3,z_1,z_2])$$ of the type
$$
\begin{array}{r}
     z_1^2 -  y_1f^2 + (l_1f +l_2y_2+l_3g)y_2^2+l_4fgy_2 = 0\\
     z_2^2 -  y_3g^2 + (m_1f +m_2y_2+m_3g)y_2^2+m_4fgy_2 = 0,
\end{array}
$$
where $$f=y_1+\alpha y_3,\ \ \ g=\beta y_1+y_3$$
and $\alpha,\beta,l_i,m_j$ are constants.
The image of the exceptional curve $E$ is the line $\{y_2=0\}$.
\end{proposition}

The involution in $T'_{6,6}$ is given by
$$
y_1\mapsto y_3,\hspace{2mm} y_2\mapsto -y_2,\hspace{2mm} y_3\mapsto y_1,\hspace{2mm} z_1\mapsto z_2,\hspace{2mm} z_2\mapsto z_1,
$$
so from here on we will only consider the parameters $\{\alpha,l_1,\dots,l_4\}$ as we set $\beta=\alpha, (m_1,m_2,m_3,m_4)=(l_3,-l_2,l_1,-l_4)$.

The reverse procedure is an unprojection of type IV (see \cite{ReidUnprojIV}).
To do so one uses a parametrization  $$\varphi\colon\P^1(u,v)\hookrightarrow \{y_2=0\}\cap T'_{6,6}$$
whose image is a genus $2$ curve which is a double cover of the image of the exceptional divisor $E$. The map is defined as
$$
(u,v)\mapsto (u^2,0,v^2,u(u^2+\alpha v^2),v(\alpha u^2+v^2)).$$

Now we explicitly describe a $3$-fold $W'_{6,6}$, projection of a $3$-fold $W$, by extending the map $\varphi$ to a map $\Phi\colon\P^2(x,u,v) \rightarrow \P(1,2^3,3^2)$ such that $\varphi(u,v)=\Phi(0,u,v)$.
We start by describing this extension in full generality, as done in \cite[$\S 4$]{CoughlanGodeaux}, i.e. extending $\varphi$ to a map $\widetilde\Phi\colon\P^5(x_1,x_2,x_3,x_4,u,v) \rightarrow \P(1^4,2^3,3^2)$ as
$$
     \widetilde\Phi^*(x_i) = x_i,\,
     \widetilde\Phi^*(y_1)  =  u^2+2x_1v,\, 
     \widetilde\Phi^*(y_2)  =  x_2u+x_3v,\, 
     \widetilde\Phi^*(y_3)  =  v^2+2x_4u.
$$
Then, to define $\Phi\colon\P^2(x,u,v) \rightarrow \P(1,2^3,3^2)$ keeping the involution, we get $x_1=x_4$ and $x_2=-x_3$. Setting $x=x_1$ and $x_2=lx$ for a constant $l$, the map $\Phi$ can be written as
$$
\begin{array}{cccc}
     \Phi^*(x) = x, & 
     \Phi^*(y_1)  =  u^2+2xv, &  
     \Phi^*(y_2)  =  lx(u-v),  & 
     \Phi^*(y_3)  =  v^2+2xu. 
\end{array}$$
We can remove the parameter $l$ by using the change of variable $y_2\mapsto ly_2$, i.e. we can consider  
$$\Phi^*(y_2)  =  x(u-v).$$
\begin{remark}
It is this change of variable, that was not used by Coughlan, that allows us to say that the family is at most 7-dimensional.
(By computing the equations without that change of variable, and then taking $l=0$,
we have checked that the particular case $l=0$ gives degenerate surfaces. Thus we can assume $l\ne 0$.)
\end{remark}
By \cite[Thm. $4.2$]{CoughlanGodeaux},
$$
\begin{array}{rcl}
     \Phi^*(z_1) & = & u(f+\alpha(1+\alpha)x^2)+(1-\alpha^2)xuv-\alpha(1-\alpha^2)x^2v, \\
     \Phi^*(z_2) & = & v(g+\alpha(1+\alpha)x^2)+(1-\alpha^2)xuv-\alpha(1-\alpha^2)x^2u.  
\end{array}
$$

We have then that $\C[x,u,v]/\C[\Phi^*(x),\Phi^*(y_i),\Phi^*(z_j)]$, as a module over the ring $\C[\Phi^*(x),\Phi^*(y_i),\Phi^*(z_j)]$, is generated by $\{1,u,v,uv\}$. In the next section we extend the surface $T'_{6,6}$ to $W'_{6,6}$, a Fano $3$-fold of index $1$, and use unprojection methods to determine the $3$-fold $W$.

\subsection{Type IV unprojection}\label{unprojIV}

Let $R=\C[x,y_1,y_2,y_3,z_1,z_2]$ be the homogeneous coordinate ring of $\P(1,2^3,3^2)$ and $\C(\Gamma)=R/I_\Gamma$, where $\Gamma$ is the image of $\Phi$. By construction, the normalisation of $\C(\Gamma)$ is the $R$-module $\C[x,u,v]$. Notice that $\C[x,u,v]$, as $R$-module, is generated by $\{1,u,v,uv\}$. Furthermore, using the embedding $\Phi^*$ to define the multiplication by elements of $R$, one can write the relations between these generators.
For $y_2$ one has
$$
\begin{array}{lcl}
y_2\cdot 1 & = & xu -xv,   \\
y_2\cdot u & = & xu^2-xuv= x(y_1-2xv)-xuv=xy_1-2x^2v-xuv,  \\
y_2\cdot v & = & xuv-xv^2=xuv-x(y_3-2xu)=-xy_3 +2x^2v + xuv, \\
y_2\cdot uv & = & x(u^2v-uv^2)=  x((y_1-2xv)v-u(y_3-2xu)) = \\
\ & = & x((y_1v-y_3u)-2x(v^2-u^2)) = \\
 & = & x((y_1v-y_3u)-2x(y_3-2xu-y_1+2xv)) = \\
 & = & 2x^2(y_1-y_3)-(xy_3-4x^3)u +(xy_1-4x^3)v.
\end{array}
$$
Doing the same for $z_1$ and $z_2$, one can write the relations in matrix form as 
$\left(\begin{array}{cccc}1 & u & v & uv\end{array}\right)B=0$, where $B$ is a $4\times 12$ matrix with entries in $R$ that can be written as $\left(\begin{array}{c|c|c}
B_{y_2} & B_{z_1}  & B_{z_2}  
\end{array}\right)$,
where $B_{y_2}, B_{z_1}, B_{z_2}$ are, respectively,
\footnotesize
$$
\begin{array}{l}
   \left(\begin{array}{cccc}
        -y_2 & xy_1 & -xy_3 &  2x^2(y_1  - y_3) \\
        x & -y_2 & 2x^2 & -xy_3+4x^3 \\
        -x & -2x^2 & -y_2 & xy_1-4x^3 \\
        0 & -x & x & -y_2 
    \end{array}\right), \\ 
    \left(\begin{array}{cccc}
        z_1 & 2s_3xy_3-y_1(f+s_1) & s_2y_3+2s_3xy_1 & 2(f+s_1)xy_3+2s_2xy_1-s_3y_1y_3 \\
        -(f+s_1) & z_1 - 4s_3x^2  & 2xs_2-y_3s_3 & 2s_3xy_1-4(f+s_1)x^2-s_2y_3 \\
        -s_2 &  2(f+s_1)x-s_3y_1  & z_1-4s_3x^2 & 2s_3xy_3-4s_2x^2-(f+s_1)y_1 \\
        -s_3 &  -s_2 & -(f+s_1)  & z_1-4s_3x^2
    \end{array}\right), \\
    \left(\begin{array}{cccc}
        z_2  & 2s_3xy_3-y_1s_2 & 2s_3xy_1-(g+s_1)y_3 & 2(g+s_1)xy_1-s_3y_1y_3+2s_2xy_3 \\
        -s_2  & z_2-4s_3x^2 & 2(g+s_1)x-s_3y_3 & 2s_3y_1x-(g+s_1)y_3-4s_2x^2 \\
        -(g+s_1)  & 2s_2x-s_3y_1 & z_2-4s_3x^2 & 2s_3xy_3-s_2y_1-4(g+s_1)x^2 \\
        -s_3  & -(g+s_1) & -s_2 & z_2 -4s_3x^2
    \end{array}\right),
\end{array} 
$$
\normalsize
and $s_1=\alpha(1+\alpha)x^2, s_2=-\alpha(1-\alpha^2)x^2, s_3=(1-\alpha^2)x$.

The matrix $B_{y_2}$ is similar to one appearing in \cite[$\S 4$]{CoughlanGodeaux} (being the last column the only difference). This matrix has the advantage that, by direct computation, these three matrices commute. One can then write the resolution of $\C[x,u,v]$ as
the Koszul resolution of a complete intersection, i.e. 
$$\C[x,u,v]\leftarrow P_0\xleftarrow{p_1} P_1\xleftarrow{p_2} P_2\xleftarrow{p_3} P_3\leftarrow 0,$$
where $p_2, p_3$ are given by the matrices 
$$
\left(\begin{array}{c|c|c}
         0 & B_{z_2} & -B_{z_1} \\ \hline
         - B_{z_2} & 0 &  B_{y_2} \\ \hline
         B_{z_1} &  -B_{y_2} & 0 
    \end{array}\right),
\left(\begin{array}{c}
         B_{y_2} \\
         B_{z_1} \\
         B_{z_2} \\
    \end{array}\right),
$$
respectively, and 
$$
\begin{array}{l}
    P_0 = R\oplus R(-1)^{\oplus 2}\oplus R(-2), \\
    P_1 = R(-2) \oplus R(-3)^{\oplus 2} \oplus R(-4) \oplus\left(R(-3) \oplus R(-4)^{\oplus 2} \oplus R(-5)\right)^{\oplus 2}, \\
    P_2 = R(-6) \oplus R(-7)^{\oplus 2} \oplus R(-8) \oplus\left(R(-5) \oplus R(-6)^{\oplus 2} \oplus R(-7)\right)^{\oplus 2}, \\
    P_3 = R(-8) \oplus R(-9)^{\oplus 2} \oplus R(-10).
\end{array}
$$

To determine the image of $\C(\Gamma)$ in $\P(1,2^3,3^2)$ one projects the graph of $\Phi$ contained in $\P(u,v,x)\times\P(1,2^3,3^2)$ into $\P(1,2^3,3^2)$. Algebraically this is just the elimination of the variables $\{u,v\}$ of the ideal $I_\Gamma$ generated by 
$$x-\Phi^*(x),\quad y_i-\Phi^*(y_i),\quad z_j-\Phi^*(z_j).$$
Computationally, such elimination turned out difficult to execute. We have succeeded only by using the software Singular with the negative degree reverse lexicographical monomial ordering (ds).
In degree $6$ we obtain six generators, $ C_1,  C_2,  Q_1,  Q_2,  Q_3,  Q_4$, which are the deformations of the polynomials
$$\widetilde C_1=z_1^2 - y_1f^2,\,\widetilde C_2=z_2^2-y_3g^2,\,\widetilde Q_1=fy_2^2,\,\widetilde Q_2=y_2^3,\,\widetilde Q_3=gy_2^2,\,\widetilde Q_4=fgy_2.$$
Although the above ordering is a local one, one can check that those polynomials are still in the ideal $I_\Gamma$.

We note that these 6 polynomials were also determined in \cite[Cor. $4.3$]{CoughlanGodeaux}.
Our approach is the one described in \cite{ReidUnprojIV}.

The $3$-fold $W'_{6,6}$ is given by the vanishing of the polynomials 
$$
\begin{array}{c}
    F:= C_1  +l_1 Q_1+l_2 Q_2+l_3 Q_3+l_4 Q_4 \\
    G:= C_2  +l_3 Q_1-l_2 Q_2+l_1 Q_3-l_4 Q_4.
\end{array}
$$

To find the unprojection variables we need the maps between the $R$ resolutions of $\C(W'_{6,6})$ and $\C[u,v,x]$.
$$
\xymatrix{
\C(W'_{6,6})\ar@{^{(}->}[d] & \ar[l] R \ar@{=}[d] &  \ar[l] R(-6)^{\oplus 2} \ar[d] & \ar[l] R(-12) & \ar[l] 0 & \\
\C(\Gamma) \ar@{^{(}->}[d] & \ar[l] R \ar@{^{(}->}[d] & \ar[l] K_1 \ar[d] & \ar[l] \cdots & & \\
\C[u,v,x] & \ar[l] P_0 & \ar[l] P_1   & \ar[l] P_2 & \ar[l] P_3 & \ar[l] 0 
}
$$
where the free module $K_1$ contains $R(-6)^{\oplus 6}$ as a direct summand corresponding to the six generators of $I_{\Gamma}$, $\left( C_1, C_2, Q_1, Q_2, Q_3, Q_4\right)$. The down arrow $R(-6)^{\oplus 2}\rightarrow K_1$ is a matrix with the following first columns
$$
\left(\begin{array}{cccccc}
    1 & 0 & l_1 & l_2 & l_3 & l_4 \\
    0 & 1 & l_3 & -l_2 & l_1 & -l_4
\end{array}\right)^t.
$$
The second down arrow, $K_1\rightarrow L_1$, expresses the generators of $K_1$ as linear combinations of the columns of $B$. Composing the maps we get the following diagram 
$$
\xymatrix{
\C(W'_{6,6})& \ar[l] R \ar@{^{(}->}[d] &  \ar[l] R(-6)^{\oplus 2} \ar[d]^{N_1} & \ar[l] \ar[d]^{N_2} R(-12) & \ar[l] 0 & \\
\C[u,v,x] & \ar[l] P_0 & \ar[l] P_1  & \ar[l] P_2 & \ar[l]_{p_3} P_3 & \ar[l] 0 
}
$$
Heavy computations are needed in order to compute the maps $N_1$ and $N_2$,
we give the details in an appendix, available as an arXiv ancillary file, containing also the correspondent Magma computations. The main idea is as follows.

Let $p$ be a matrix and $H$ be a vector, both with entries in a multivariate polynomial ring. One can use the Magma function 'Solution' to compute $N$ such that $pN=H.$ But in the case we are interested in, this computation finishes only if we fix the parameter $\alpha.$ So we define a Magma function that does it several times by evaluating $\alpha$ at a list $pts$ of points, then we use these data to recover the coefficients of the computed polynomials as polynomials on the parameter $\alpha$.

Since we are fixing one of the parameters, it can happen that this value appears in some denominator of a rational coefficient of the solution. To overcome this, we have included an input polynomial 'correction': after each computation of $N$ for $\alpha=\alpha_0\in pts,$ the solution is multiplied by that polynomial evaluated at $\alpha_0$.

Having the matrix $N_2,$ one can write the linear equations of the unprojection of $\Gamma$ in $W'_{6,6}$. The method is similar to the unprojection of type I. As $\Gamma$ is a codimension $1$ subscheme of $W'_{6,6}$, using adjunction formula one has
\begin{equation}\label{unprojses}
    0\longleftarrow \omega_\Gamma\longleftarrow \HHom_{\O_{W'_{6,6}}}\left(I_\Gamma,\omega_{W'_{6,6}}\right)\longleftarrow \omega_{W'_{6,6}}\longleftarrow 0.
\end{equation}

As $\C(\Gamma)\hookrightarrow\C[u,v]$ is an isomorphism outside the origin, the dualising module satisfies $\omega_\Gamma=\omega_{\C[x,u,v]}\cong \C[x,u,v]$. On the other hand, as a module over $\O_{W'_{6,6}}$, or over $R$, it needs $4$ generators, $\{1,u,v,uv\}$.

The coordinate ring of the unprojected variety is obtained from $\C(W'_{6,6})$ by adjoining rational functions $\{y_4,z_3,z_4,t\}$ with poles along $\Gamma$. These can be seen as homomorphisms in 
$\HHom_{\O_{W'_{6,6}}}\left(I_\Gamma,\omega_{W'_{6,6}}\right)$. 

The variable $y_4$ is the rational form that maps to a basis of $\omega_\Gamma(3)\cong \O_\Gamma=\O_{\P^2}$. Notice that $\omega_{W'_{6,6}}\cong\O_{W'_{6,6}}(-1)$ hence, using sequence (\ref{unprojses}), we get that $\deg(y_4)=2$. Denoting by $z_3, z_4, t$ the forms that map to $u, v, uv$, respectively, we get $\deg(z_i)=3$, $\deg(t)=4$.

The linear relations between $y_4,z_3,z_4,t$ are given as in a type I unprojection. Each of them corresponds to a generator of $P_3$ and is mapped by $p_3$ to the image of $N_2$, i.e.
$$
p_3\left(\begin{array}{c}
     t  \\ z_4 \\ z_3 \\ y_4
\end{array}\right)=N_2.
$$

We have now to determine the quadratic relations between the unprojection variables $\{z_3,z_4,t\}$. Notice that the extension $\C(W'_{6,6})\subset \C(W)$ can be seen as the normalisation of the ring $\C(W'_{6,6})[y_4]$ or, as in some sense $z_3, z_4, t$ correspond to $uy_4, vy_4, uvy_4$, respectively,  there must be relations of the form
$$
\begin{array}{rcl}
    z_3^2 -y_1y_4^2+2xy_4z_4,\, z_3z_4-y_4t,\,  z_4^2-y_3y_4^2+2xy_4z_3 & \in & \langle \text{mon. of degree } 6\rangle, \\
    z_3t-y_1y_4z_4+2xz_4^2,\, z_4t-y_3y_4z_3+2xz_3^2 & \in & \langle \text{mon. of degree } 7\rangle, \\
    t^2-y_1y_3y_4^2+2x(y_1y_4z_3+y_3y_4z_4)-4x^2y_4t & \in & \langle \text{mon. of degree } 8\rangle, \\
\end{array}
$$
where the monomials on the right hand side are linear in the unprojection variables. 

We will find these relations as equations $f=0$ with $xf$ or $y_2f$ contained in the ideal generated by the linear equations $F_1,\ldots,F_{14}$.
The detailed Magma computation is given in the appendix referred above, the idea is as follows.
Let $G_i,H_i$ be such that $F_i=xG_i+H_i.$ A polynomial $\sum c_iF_i$ is divisible by $x$ if $\sum c_iH_i=0.$
In order to find such coefficients $c_i,$
it suffices to compute the syzygy matrix of the sequence $E:=[E_1,\ldots,E_{14}]$ obtained by evaluating the equations $H_i$ at $x=0.$
But our computer cannot finish this, so we replace the parameters $\alpha,l_j$ appearing in the $E_i$ by some distinct prime numbers.
In this way we can compute the syzygy matrix, obtaining a list of relations of the type $$c_1E_1+\cdots +c_{14}E_{14}=0.$$
Each $c_i$ is a polynomial in the variables $y_1,y_2,y_3,z_1,z_2,z_3,z_4,t,$ with coefficients in $\mathbb Q.$
We want to recover these coefficients as polynomials in $\alpha,l_j.$
For that aim, we replace each coefficient by a new variable, obtaining a linear system (with a lot of variables) that we can solve
(in fact the computations turned out to be simple, except for one of the polynomials).

\begin{remark}
On the unprojection method to describe $W$ no new parameters were used, so they are the parameters describing $W'_{6,6}$, i.e. $\{\alpha,l_1,l_2,l_3,l_4\}$.
\end{remark}

\begin{remark}
The bicanonical image of a general surface $Y$ is an octic in $\mathbb P^3$, see \cite{CataneseY}.
We have computed the family of these octics for Coughlan's family of surfaces by eliminating the variables $z_1,\ldots,z_4,t$ from the equations.
This octic polynomial is given in an ancillary arXiv file.
\end{remark}

\subsection{Description of $Y$ and parameter counting}\label{count}

\begin{proposition}\label{parametercount}
The Coughlan family of Godeaux surfaces with $\pi_1=\Z/2$, where each $X$ is obtained as a $\Z/2$ quotient of a hyperelliptic surface $Y$ such that $K^2=2, p_g=1, q=0$, is determined by $7$ parameters. 
\end{proposition}

\begin{proof}
We have a hyperelliptic tower $D\subset T\subset W$, and in Section \ref{unprojIV} we described the ring $$R(W,-K_W)=\C[x,y_1,y_2,y_3,y_4,z_1,z_2,z_3,z_4,t]/I.$$
Furthermore, we have an involution $\sigma\colon W\rightarrow W$ whose action on $\P(1,2^4,3^4,4)$ has the following eigenspaces 
$$
\begin{array}{c|c|c}
    n & H^0(W,-nK_W)^+ & H^0(W,-nK_W)^- \\ \hline
    1 &   & x \\
    2 & y_1+y_3 & y_1-y_3, y_2, y_4 \\
    3 & z_1-z_2, z_3+z_4 & z_1+z_2,z_3-z_4 \\
    4 & & t
\end{array}
$$

A model for the canonical ring of the surfaces $Y$, is then given by 
$$R(Y,K_Y)=R(W,-K_W)/(H_2^-),$$
where $H_2^-$ is a hyperplane of degree $2$ that is $\sigma$-anti-invariant.
By construction, $W$ depends on the parameters $\{\alpha, l_1,\dots,l_4\}$.
As $\dim\left(\P(H^0(W,-2K_W)^-)\right)=2$, we get a total of $7$ parameters. In this way, there is a redundant parameter in \cite[Theorem $1$]{CoughlanGodeaux}.
\end{proof}

\begin{remark}
The counting could have been done in a different way. In the description of $T_{6,6}'$ we get the variables $y_1, y_3$ fixed. On the other hand $y_2$ is only defined as the variable such that $\{y_2=0\}$ is the line that goes through the nodes of the two plane cubics, hence we are free to re-scale $y_2$ at will. With this is mind, one can remove one of the parameters $l_1,\ldots,l_4$ and see that $T'_{6,6}$ depends on $4$ parameters. Doing so, one gets one parameter for the extension proving that for each $K3$ we have a set of $3$-folds parametrized by a single parameter. 
\end{remark}

\subsection{Families of universal covers of $\Z/2$-Godeaux surfaces} \label{16dim}

Let $\mathscr M$ be the moduli space of simply connected surfaces with $p_g=1,$ $q=0$ and $K^2=2$,
and let $\mathscr M_1$ be the subvariety corresponding to surfaces with bicanonical map of degree
$4$ onto a smooth quadric surface in $\mathbb P^3.$
Catanese and Debarre \cite{CataneseY} have shown that there is a unique $16$-dimensional irreducible
component $\mathscr M_E\subset \mathscr M$ which contains $\mathscr M_1.$

\begin{proposition}
Coughlan's family of surfaces with $p_g=1,$ $q=0$ and $K^2=2$ is contained in $\mathscr M_E$
(thus the topological fundamental group of Coughlan's $\Z/2$-Godeaux surfaces is indeed $\mathbb Z/2$).
\end{proposition}

\begin{proof}
The smooth $K3$ surfaces $T$ are flat deformations of the canonical curve $D\subset Y$ by a regular element $y_4$ of degree $2$. Furthermore, by construction these surfaces project into $\{(y_1+\alpha y_3)(\alpha y_1+y_3)-y_2y_4=0\}$. 

With the change of variable $y_4=x^2$, one gets a component of the flat extensions of $D$ by a variable of degree $1$, i.e. surfaces with the invariants $K^2=2, p_g=1, q=0$. These surfaces project into the smooth quadric $$\{(y_1+\alpha y_3)(\alpha y_1+y_3)-x^2y_2=0\}\subset \mathbb{P}^2(1,2,2,2).$$

To see this family one neglects the extension of the embedding $\varphi\colon \mathbb P(u,v)\rightarrow T'_{6,6}$ to $\widetilde\Phi\colon \mathbb P(x,u,v)\rightarrow W'_{6,6}$. Or, in other words, one sets all the \mbox{parameters} \mbox{describing} such extension to be zero. To be more specific, recall that the \mbox{extension} of the embedding (for the non-involution case) is defined as
$$
     \widetilde\Phi^*(x_i) = x_i,\,
     \widetilde\Phi^*(y_1)  =  u^2+2x_1v,\, 
     \widetilde\Phi^*(y_2)  =  x_2u+x_3v,\, 
     \widetilde\Phi^*(y_3)  =  v^2+2x_4u.
$$
Setting each $x_i=a_ix$, where $a_i$ is a parameter, Coughlan's family is unirational and parametrized by $\{\alpha,\beta,l_i,m_i,a_i,c_i\}$, where the $c_i$ are the parameters defining the hyperplane 
$$\{y_4-c_0x^2-c_1y_1-c_2y_2-c_3y_3=0\}.$$
Then any member of Coughlan's family can be deformed to a surface in $\mathscr{M}_1$ by linearly mapping $a_i\mapsto 0$ and 
$(c_0,c_1,c_2,c_3)\mapsto (1,0,0,0).$

\end{proof}

\begin{remark}
There exists an 8-dimensional family $\mathcal M$ of $\Z/2$-Godeaux surfaces whose universal covers live in a 16-dimensional family, so that an 8-dimensional subvariety of this family parametrizes surfaces with a $\Z/2$ free action whose quotients give back $\mathcal M$. To see this, we consider the example $W$ of type B6 constructed in \cite[\S 3.1 (C)]{RU19}. It has one Wahl singularity $[2,2,3,5,4]$, and it has no-local-to-global obstructions to deform. $\Q$-Gorenstein smoothings of $W$ produce an 8-dimensional family of Godeaux surfaces with fundamental group $\Z/2$. To prove unobstructedness, we show (see Lemma \ref{trick}) that the \'etale double cover $W'$ of $W$ induced by the \'etale double cover of the Enriques surface  has also no-local-to-global obstructions in deformation, and so it produces a 16-dimensional family of simply-connected surfaces of general type with $K^2=2$, and $p_g=1$. Since Pic$(W) \subset$ Pic$(X)$, where $X$ is a $\Q$-Gorenstein smoothing of $W$, we have a lifting of the \'etale cover on a subfamily of $Y$'s ($\Q$-Gorenstein smoothings of $W'$). This is similar to the procedure used in \cite{PSU13} for a branched double cover. In this way, we would expect that this 16-dimensional family of simply connected surfaces is $\mathscr M_E$, where an 8-dimensional subfamily gives the moduli space of $\Z/2$-Godeaux surfaces. We leave it as an open question. (Of course, we could have started with another $\Q$-Gorenstein degeneration.)
\label{qgor}
\end{remark}

\section{Degenerations from Coughlan's family}\label{CoughlanDnm}

The computations below were implemented with Magma \cite{BCP}, version V2.25-2, and are available as arXiv ancillary files.

\subsection{$D_{2,4}$ elliptic surfaces}\label{D24}

Denote by \(Y\subset\mathbb P=\mathbb P(1,2,2,2,3,3,3,3,4)\) a surface in Coughlan's family,
whose general element is the universal covering of a \(\mathbb Z/2\)-Godeaux surface.
Our computer experiments over finite fields say that there are values of the parameters ($\alpha,l_1,l_2,l_3,l_4,l_5,l_6$) for which the
corresponding surface \(Y\hookrightarrow \mathbb P^7\) (for a general $Y$ this is the embedding by the $3$-canonical map)
splits as the union of a degree 16 surface \(Y'\) with two planes,
and the quotient of \(Y'\) by the "Godeaux" involution is a \(D_{2,4}\) elliptic surface. Those two planes correspond to two base points of the map \(\mathbb P\rightarrow\mathbb P^7\),
and the coordinates of these points satisfy \(x=z_1=z_2=z_3=z_4=t=0\).

Here we use this information to obtain a 6-dimensional family of
\(D_{2,4}\) elliptic surfaces as quotients of a codimension 1 subset of
Coughlan's family of surfaces.
\\

\noindent{\bf Step 1.}\\
We load the \(20\) equations that define Coughlan's family, evaluated at \(x=z_1=z_2=z_3=z_4=t=0\).
Then we impose \(y_1+y_3\neq 0\) and eliminate all variables except the parameters.
We get one single relation $f=0$ on the parameters.
Our goal is to show that a random point in this set of parameters corresponds to a surface such that its quotient
by the "Godeaux" involution is a \(D_{2,4}\) elliptic surface with a \((-4)\)-curve.\\

\noindent{\bf Step 2.}\\
We take such a random surface and want to embed it in $\mathbb P^7$ with coordinates
$(X_0,\ldots,X_7)=(x^3,xy_1,xy_2,xy_3,z_1,z_2,z_3,z_4)$.
To achieve this we have to eliminate the variable $t$, but the computer cannot do it.
We have done it "by hand", obtaining a set of equations that is not complete:
for general values of the parameters we get a surface plus the component $X_0=0.$
Removing this component we get a surface $Y'.$

To speed up the computations, we work over a finite field.
\begin{remark}
Working over a field of characteristic zero, we can compute the complement of the canonical curve of $Y'$
in its hyperplane $X_0=0$, obtaining the union of two disjoint rational curves.
Moreover, these meet the canonical curve with multiplicity 2, hence are $(-4)$-curves,
and are identified by the fixed point free "Godeaux" involution.
Thus it follows from Theorem 2.1 that the quotient of $Y'$ by the involution is a $D_{2,n}$ elliptic surface.
We will show that $n=4$ by computing, over a finite field, the elliptic fibres of multiplicities 2 and 4.
This implies that $n=4$ also over the base field $\mathbb C.$ 
\end{remark}

\noindent{\bf Step 3.}\\
In order to compute the singular subscheme of \(Y'\), we need first to reduce the number of its defining equations.
We wrote an algorithm for that, which basically removes one equation at a time.
Then we verify that \(Y'\) is smooth.
\\

\noindent{\bf Step 4.}\\
We check that the "Godeaux" involution acts freely on \(Y'\).
\\

\noindent{\bf Step 5.}\\
The hyperplane \(X_0=0\) cuts \(Y'\) at the union of the canonical divisor of \(Y'\) with two disjoint \((-4)\)-curves,
which are identified by the "Godeaux" involution.
\\

\noindent{\bf Step 6.}\\
Some Magma functions give the invariants of \(Y'\).
\\

\noindent{\bf Step 7.}\\
Studying the equations of the pencil \(|2K_Y'|\),
we find elliptic curves \(D_1\), \(D_2\) such that \(D_1\equiv 2D_2\).
We see that these two curves are fixed by the (fixed point free) "Godeaux" involution. This shows that the quotient of \(Y'\) by the involution is a \(D_{2,4}\) elliptic surface with one \((-4)\)-curve. \\

\subsection{$D_{2,3}$ elliptic surfaces}\label{D23}

Denote by \(Y\) an element of Coughlan's family of surfaces, whose
general surface is the universal covering of a \(\mathbb Z/2\)-Godeaux
surface. Our computer experiments over finite fields say that there are
values of the parameters ($\alpha,l_1,l_2,l_3,l_4,l_5,l_6$) for which the surface
\(Y\subset\mathbb P(1,2,2,2,3,3,3,3,4)\) contains a node, which is the
only point that is fixed by the ``Godeaux" involution. Moreover, the
smooth minimal model of the quotient of \(Y\) by that involution is a
\(D_{2,3}\) elliptic surface. The coordinates of that point satisfy
\(y_2=y_3-y_1=z_2-z_1=z_4+z_3=0\).

Here we use this information to obtain a 6-dimensional family of
\(D_{2,3}\) elliptic surfaces as quotients of a codimension 1 subset of Coughlan's family of surfaces.\\

\noindent{\bf Step 1.}\\
We load the \(20\) equations that define Coughlan's family, and we
impose \(y_2=0\), \(y_3=y_1\), \(z_2=z_1\) \(z_4=-z_3\). Then we eliminate all variables except the parameters.
(To speed up computations, we fix the parameter $\alpha$.)
We obtain one single relation, which contains the component $l_1+l_3=0$ (which does not depend on $\alpha$).
\\

\noindent{\bf Step 2.}\\
We pick an arbitrary surface on the family given by \(l_1+l_3=0\). We aim to show that the resolution of its quotient
by the "Godeaux" involution is indeed a \(D_{2,3}\) elliptic surface with a \((-4)\)-curve.
\\

\noindent{\bf Step 3.}\\
We check that the subscheme of \(Y\) that satisfies \(y_2=y_3-y_1=z_2-z_1=z_4+z_3=0\) is
a point, which is a node fixed by the involution.
Thus it follows from Theorem 2.1 that the quotient of $Y$ by the involution is a $D_{2,n}$ elliptic surface.
In order to speed up the computations, from now on we work over a finite field. We will show that $n=3$ by computing
the $D_{2,3}$ surface and its double and triple fibres. This implies that $n=3$ also over the base field $\mathbb C.$ 
\\

\noindent{\bf Step 4.}\\
We compute the linear system of the curves of degree 5 that contain
the above fixed point and are preserved by the "Godeaux" involution.
This system defines a map \(\phi:Y \to \mathbb P^{10}\),
which resolves the singularity of \(Y\).
We will show that it is of degree 2 onto a \(D_{2,3}\) elliptic surface $G$ with a \((-4)\)-curve that is the image of the node of \(Y\).
\\

\noindent{\bf Step 5.}\\
The direct computation of \(\phi(Y)\) seems unattainable, so we compute
the image of many points and then the linear systems \(L_2,\) \(L_3\)
of hypersurfaces of degree 2, 3 through these points. These cut out a surface $G$ in $\mathbb P^{10}.$
\\

\noindent{\bf Step 6.}\\
In order to show that \(G\) is smooth, and to avoid the computation of
all \(8\times 8\) minors of matrices of partial derivatives, we random
such minors until they define an empty subscheme of \(G\).
\\

\noindent{\bf Step 7.}\\
Some Magma functions give the invariants of $G$.
\\

\noindent{\bf Step 8.}\\
The system $2K_Y$ is given by the pullback of $2K_G+C,$ where $C$ is the $(-4)$-curve corresponding to the node of $Y$.
This means that there exists an invariant bicanonical curve through the node of $Y$.
We show that its quotient in $G$ is $H:=F_3+C$, where $3F_3$ is an elliptic fibre and $C$ is a $(-4)$-curve.
\\

\noindent{\bf Step 9.}\\
We find the double elliptic fibre \(2F_2\) by computing the unique element in $|F_3+K_G|$.
\\

\noindent{\bf Step 10.}\\
Finally we check that $CF_3=4$ and $C$ is the image of the node of $Y$.


\vspace{0.5cm}

\noindent Eduardo Dias
\vspace{0.1cm}
\\ Departamento de Matem\' atica
\\ Faculdade de Ci\^encias da Universidade do Porto
\\ Rua do Campo Alegre 687
\\ 4169-007 Porto, Portugal
\\ www.fc.up.pt, {\tt eduardo.dias@fc.up.pt}\\

\vspace{0.3cm}

\noindent Carlos Rito
\vspace{0.1cm}
\\{\it Permanent address:}
\\ Universidade de Tr\'as-os-Montes e Alto Douro, UTAD
\\ Quinta de Prados
\\ 5000-801 Vila Real, Portugal
\\ www.utad.pt, {\tt crito@utad.pt}
\vspace{0.1cm}
\\{\it Temporary address:}
\\ Departamento de Matem\' atica
\\ Faculdade de Ci\^encias da Universidade do Porto
\\ Rua do Campo Alegre 687
\\ 4169-007 Porto, Portugal
\\ www.fc.up.pt, {\tt crito@fc.up.pt}\\

\vspace{0.3cm}

\noindent Giancarlo Urz\'ua
\vspace{0.1cm}
\\ Facultad de Matem\'aticas
\\ Pontificia Universidad Cat\'olica de Chile
\\ Campus San Joaqu\'in
\\ Avenida Vicu\~na Mackenna
\\ 4860, Santiago, Chile
\\ {\tt urzua@mat.uc.cl}\\

\end{document}